\documentclass{amsart}
\usepackage{amssymb,graphicx}

\title[Fast loops on semi-weighted homogeneous hypersurface singularities]
{Fast loops on semi-weighted homogeneous hypersurface singularities}
\author{Alexandre Fernandes}\thanks{This paper was written during my posdoctoral stage at Universidade de S\~ao Paulo and this
research was partially supported by CNPq grant N 150578/2009-1.}
\address{Departamento de Matem\'atica, Universidade Federal
do Cear\'a, Av. Mister Hull s/n,Campus do PICI, Bloco 914,CEP:
60.455-760 - Fortaleza - CE - Brasil.}
\email{alexandre.fernandes@ufc.br}

\date{\today}



\def \N {\mathbb{N}}
\def \Q {\mathbb{Q}}

\def \R {\mathbb{R}}
\def \C {\mathbb{C}}

\newtheorem{theorem}{Theorem}[section]
\newtheorem{lemma}[theorem]{Lemma}
\newtheorem{proposition}[theorem]{Proposition}
\newtheorem{corollary}[theorem]{Corollary}
\theoremstyle{definition}

\newtheorem{example}[theorem]{Example}
\theoremstyle{remark}

\numberwithin{equation}{section}

\begin{document}

\begin{abstract}
We show the existence of ($1+\frac{w_2}{w_3}$)-fast loops on semi-weighted homogeneous hypersurface singularities with weights $w_1\geq w_2>w_3$.
In particular we show that semi-weighted homogeneous hypersurface singularities have metrical conical structure only if
its two low weights are equal.
\end{abstract}

\maketitle

\section{Introduction}

Let $X\subset\R^n$ be a subanalytic set with a singularity at $x$. It is well-known for small real numbers
$\epsilon>0$ that there exists a homeomorphism from the Euclidean ball $B(x,\epsilon)$ to itself which maps
$X\cap B(x,\epsilon)$ onto the straight cone over $X\cap S(x,\epsilon)$ with vertex at $x$. The homeomorphism $h$ is called a \emph{topological
conical structure} of $X$ at $x$ and, since John Milnor proved
the existence of topological conical structure for algebraic complex hypersurfaces with an isolated singularity \cite{M},
some authors say $\epsilon$ is a \emph{Milnor radius} of $X$ at $x$. Some developments of a Lipschitz geometry of complex
algebraic singularities come from the following question: given an algebraic subset $X\subset\C^n$ with an isolated singularity at
$x$, is there $\epsilon>0$ such that $X\cap B(x,\epsilon)$ is bi-Lipschitz homeomorphic to the cone over $X\cap S(x,\epsilon)$ with vertex at $x$? When we have a positive answer for this question we say that \emph{$(X,x)$ admits a  metrical conical structure}. Some motivations
for this question were given in \cite{BF}, \cite{BFN} and, in the same papers, the above question was answered negatively. The strategy used in \cite{BFN} to show that some examples of complex algebraic surface singularities do not admit a metrical conical structure was estimating a homotopic version of the first characteristic exponent for weighted homogeneous surface singularities (see \cite{BB} or \cite{BC} for a definition of characteristic exponent). In this paper we compute this exponent for some semi-weighted homogeneous surface singularities.

\section{Preliminaries}

\subsection{Inner metric}\label{section_inner_metric}
Given an arc $\gamma\colon[0,1]\rightarrow\R^n$, we remember that the {\it length of} $\gamma$ is defined by
$$l(\gamma)=\inf\{\sum_{i=1}^{m} |\gamma(t_{i})-\gamma(t_{i-1})| \ : \ 0=t_0<t_1<\cdots<t_{m-1}<t_m=1 \}.$$

Let $X\subset\R^n$ be a subanalytic connected subset. It is well-know that the function
$$d_X\colon X\times X\rightarrow [0,+\infty)$$
defined by
$$d_X(x,y)=\inf\{l(\gamma) \ : \ \gamma\colon [0,1]\rightarrow X; \ \gamma(0)=x , \ \gamma(1)=y\}$$ is a metric on $X$, so-called {\it inner metric} on $X$.

\begin{theorem}[Pancake Decomposition \cite{K}] Let $X\subset\R^n$ be a subanalytic connected subset. Then, there exist $\lambda>0$
 and $X_1,\dots,X_m$ subanalytic subsets such that:
\begin{enumerate}
\item[a.] $X=\bigcup_{i=1}^{m}X_i$,
\item[b.] $d_X(x,y)\leq\lambda |x-y|$ for any $x,y\in X_i$, $i=1,\dots,m$.
\end{enumerate}
\end{theorem}

\subsection{Horn exponents}\label{section_horn_exponents}

Let $\beta \geq 1$ be a rational number. The germ of
$$H_{\beta}=\{ (x,y,z)\in \R^3 \ : \ x^{2}+y^{2}=z^{\beta}, \ z\geq 0 \}$$ at origin $\in\R^3$
is called a {\it $\beta$-horn}.

By results of \cite{B}, we conclude that a ${\beta_1}$-horn is bi-Lipschitz equivalent, with respect to the inner
metric, to a ${\beta_2}$-horn if, and only if $\beta_1=\beta_2$. Let $\Omega\subset \R^n$
be a 2-dimensional subanalytic set. Let $x_0\in \Omega$ be a point
such that $\Omega$ is a topological 2-dimensional manifold without
boundary near $x_0$.

\begin{theorem} \cite{B} \label{horn-theorem}
There exists a unique rational number $\beta\geq 1$ such that the
germ of $\Omega$ at $x_0$ is bi-Lipschitz equivalent, with respect
to the inner metric, to a $\beta$-horn.
\end{theorem}

The number $\beta$ is called {\it the horn exponent of $\Omega$ at
$x_0$}. We use the notation $\beta(\Omega,x_0)$. By Theorem
\ref{horn-theorem}, $\beta(\Omega,x_0)$ is a complete intrinsic
bi-Lipschitz invariant of germ of subanalytic sets which are
topological 2-dimensional manifold without boundary. In the following, we
show a way to compute horn exponents.

\bigskip

According to \cite{BB},
$\beta(\Omega,x_0)+1$ is the volume growth number of $\Omega$ at $x_0$, i. e.
$$\beta(\Omega,x_0)+1=\lim_{r\to 0+}\frac{\log \mathcal{H}^2[\Omega\cap B(x_0,r)]}{\log r}$$
where $\mathcal{H}^2$ denotes the $2$-dimensional Hausdorff measure with respect to Euclidean metric on $\R^n$.

\subsection{Order of contact of arcs}

Let $\gamma_1\colon [0,\epsilon)\rightarrow \Omega$ and
$\gamma_2\colon [0,\epsilon)\rightarrow \Omega$ be two continuous semianalytic
arcs with $\gamma_1(0)=\gamma_2(0)=x_0$ and not identically equal to
$x_0$. We suppose that the arcs are parameterized in the following
way:
$$\|\gamma_i(t)-x_0\|=t, \ i=1,2.$$
Let $\rho(t)$ be a function defined as follows:
$\rho(t)=\|\gamma_1(t)-\gamma_2(t)\|$. Since $\rho$ is a subanalytic
function there exist numbers $\lambda \in \Q$ and $a\in \R$, $a\neq
0$, such that
$$\rho(t)=at^{\lambda}+o(t^{\lambda}).$$
The number $\lambda$ is called {\it an order of contact of
$\gamma_1$ and $\gamma_2$}. We use the notation
$\lambda(\gamma_1,\gamma_2)$ (see \cite{BF1}).

Let $K$ be the field of germs of subanalytic functions
$f\colon(0,\epsilon)\rightarrow \R$. Let $\nu\colon K\rightarrow\R$
be a canonical valuation on $K$. Namely, if $f(t)=\alpha
t^{\beta}+o(t^{\beta})$ with $\alpha\neq 0$ we put $\nu(f)=\beta$.

\begin{lemma}\label{comparison2} Let $\gamma_1,\gamma_2$ be a pair of
semianalytic arcs such that $\gamma_1(0)=\gamma_2(0)=x_0$ and
$\gamma_i\neq x_0$ ($i=1,2$). Let $\tilde{\gamma_1}(\tau)$ and
$\tilde{\gamma_2}(\tau)$ be semianalytic parameterizations of
$\gamma_1$ and $\gamma_2$ such that
$\|\tilde{\gamma_i}(\tau)-x_0\|=\tau + o_i(\tau)$, $i=1,2$. Let
$l(\tau)=\|\tilde{\gamma_1}(\tau)-\tilde{\gamma_2}(\tau)\|$. Then
$\nu(l(\tau))\leq \lambda(\gamma_1,\gamma_2)$.
\end{lemma}

The following result is an alternative way to compute horn exponents of germ of subanalytic sets which are
topological 2-dimensional manifold without boundary.

\begin{theorem}\label{calculation_of_beta(horn)}
Let $\Omega\subset\R^n$ be a 2-dimensional subanalytic set. Let
$x_0\in \Omega$ be a point such that $\Omega$ is a topological
2-dimensional manifold without boundary near $x_0$. Then
$\beta(\Omega,x_0)=\inf \{\lambda(\gamma_1,\gamma_2) \ : \
\gamma_1,\gamma_2 \mbox{ are semianalytic arcs on $\Omega$ with }
\gamma_1(0)=\gamma_2(0)=x_0\}$.
\end{theorem}

Lemma \ref{comparison2} and Theorem \ref{calculation_of_beta(horn)} were proved in \cite{BF3}.

\section{Fast loops}
Let $X\subset\R^n$ be a subanalytic set with a singularity at $x$. Let $\epsilon >0$ be a Milnor radius of $X$ at $x$ and let us denote by $X^*$ the set $X\cap B(x,\epsilon)\setminus \{ x\}$. Given a positive real number $\alpha$, a continuous map $\gamma\colon S^1\rightarrow X^*$ is called a  $\alpha$-fast loop if there exists a homotopy $H\colon S^1\times [0,1]\rightarrow X\cap B(x,\epsilon)$ such that
\begin{enumerate}
\item $H(\theta,0)=x$ and $H(\theta,1)=\gamma(\theta)$, $\forall \ \theta\in S^1$,
\item $\displaystyle\lim_{r\to 0^+}\frac{1}{r^a}\mathcal{H}^2(Im(H)\cap B(x,r))=0$ for each $0<a<\alpha$,
\end{enumerate}
where $Im(H)$ denotes the image of $H$.

Given a subanalytic set $X$ and a singular point $x\in X$, according to \cite{BB}, there exists a positive number $c$ such that any $\alpha$-fast loop $\gamma\colon S^1\rightarrow X^*$ with $\alpha>c$ is necessarily homotopically trivial. Such a number $c$ is called \emph{distinguished for} $(X,x)$. We define the $\upsilon$ invariant in the following way:

$$\upsilon(X,x)=\inf\{c \ : \ c \ \mbox{is distinguished for} \ (X,x)\}.$$

The number $\upsilon(X,x)$ defined above is a homotopic version of the first characteristic exponent for the local metric homology presented in \cite{BB} .

\begin{example} Let $K\subset\R^n$ be a straight cone over a Nash submanifold $N\subset\R^n$, with vertex at $p$. Then every loop $\gamma\colon S^1\rightarrow K^*$ is a $2$-fast loop. Moreover, if $\alpha>2$, then each $\alpha$-fast loop  $\gamma\colon S^1\rightarrow K^*$ is homotopically trivial. We can sum up it saying $\upsilon(K,p)=2$.
\end{example}

\begin{proposition}\label{invariance}
Let $(X,x)$ and $(Y,y)$ be subanalytic germs. If there exists a germ of a bi-Lipschitz homeomorphism, with respect to inner metric, between $(X,x)$ and $(Y,y)$, then $\upsilon(X,x)=\upsilon(Y,y)$.
\end{proposition}

\begin{proof} Let $f\colon (X,x)\rightarrow (Y,y)$ be a bi-Lipschitz homeomorphism, with respect to the inner metric. Given $A\subset X$, let us denote $\tilde{A}=f(A)$. In this case, $A=f^{-1}(\tilde{A})$, where $f^{-1}$ denotes the inverse map of $f\colon (X,x)\rightarrow (Y,y)$.

\bigskip

\noindent{\tt Claim.} There are positive constants $k_1,k_2,\lambda_1,\lambda_2$ such that
$$\frac{1}{k_1}\mathcal{H}^2(\tilde{A}\cap B(y,\frac{r}{\lambda_2}))\leq \mathcal{H}^2(A\cap B(x,r))
\leq k_2\mathcal{H}^2(\tilde{A}\cap B(y,\lambda_1 r)).$$

In fact, using Pancake Decomposition Theorem (see Subsection \ref{section_inner_metric}) and using that $f$ and $f^{-1}$ are Lipschitz maps, we obtain positive constants $\lambda_1,\lambda_2$ such that $$f(A\cap B(x,r))\subset (\tilde{A}\cap B(y\lambda_1r)) \ \mbox{and} \
f(\tilde{A}\cap B(y,r))\subset (A\cap B(x\lambda_2r))$$ and we also obtain positive constants $k_1,k_2$ such that
$$\mathcal{H}^2(f(A\cap B(x,r)))\leq k_1 \mathcal{H}^2(A\cap B(x,r)) \ \mbox{and} \
\mathcal{H}^2(f^{-1}(\tilde{A}\cap B(y,r)))\leq k_2 \mathcal{H}^2(\tilde{A}\cap B(y,r)).$$ Our claim follows from this two inequalities and the two inclusions above.

Now, we use this claim to show that given $\alpha>0$, a loop $\gamma\colon S^1\rightarrow X\setminus \{ x\}$ is an $\alpha$-fast loop if, and only if, $f\circ\gamma\colon S^1\rightarrow Y\setminus \{ y\}$ is an $\alpha$-fast loop. In fact, let  $\gamma\colon S^1\rightarrow X\setminus \{ x\}$ be a loop and $H\colon S^1\times [0,1]\rightarrow X$ a homotopy such that $H(\theta,0)=x$ and $H(\theta,1)=\gamma(\theta)$, $\forall \ \theta\in S^1$. Thus, $f\circ\gamma\colon S^1\rightarrow Y\setminus \{ y\}$ is a loop and $f\circ H\colon S^1\times [0,1]\rightarrow X$ is a homotopy such that $f\circ H(\theta,0)=x$ and $f\circ H(\theta,1)=f\circ\gamma(\theta)$, $\forall \ \theta\in S^1$. Let us denote $A=Im(H)$ and $\tilde{A}=Im(f\circ H)$, i. e., $\tilde{A}=f(A)$. Given $0<a<\alpha$, by the above claim, we have that
$$\lim_{r\to 0^+}\frac{1}{r^a}\mathcal{H}^2(A\cap B(x,r))=0$$ if, and only if,
$$\lim_{r\to 0^+}\frac{1}{r^a}\mathcal{H}^2(\tilde{A}\cap B(y,r))=0.$$ In another words, it was shown that $\gamma\colon S^1\rightarrow X\setminus \{ x\}$ is an $\alpha$-fast loop if, and only if, $f\circ\gamma\colon S^1\rightarrow Y\setminus \{ y\}$ is an $\alpha$-fast loop, hence $\upsilon(X,x)=\upsilon(Y,y)$.
\end{proof}

\begin{corollary}\label{invariance_corollary} Let $X\subset\R^n$ be a subanalytic set and $x\in X$ an isolated singular point. If $\upsilon(X,x)>2$, then $X$ does not admit a metrical conical structure at $x$.
\end{corollary}

\begin{proof} Let $N$ be the intersection $X\cap S(x,\epsilon)$ where $\epsilon>0$ is chosen sufficiently small. Since $x$ is an isolated singular point of $X$, we have $N\subset\R^n$ is a Nash submanifold. If $X$ has metrical conical structure at $x$, $X\cap B(x,\epsilon)$ must be bi-Lipschitz homeomorphic (with respect to the inner metric) to the straight cone over $N$ with vertex at $x$. Thus, it follows from Proposition \ref{invariance}   that $\upsilon(X,x)=2$.
\end{proof}

\section{semi-weighted homogeneous hypersurface singularities}

Remind that a polynomial function $f\colon \C^3\rightarrow \C $ is
called {\it semi-weighted homogeneous} of degree $d\in \N$ with
respect to the weights $w_1,w_2,w_3 \in \N$ if $f$ can be present in the
following form: $f=h+\theta$ where $h$ is a weighted homogeneous
polynomial of degree $d$ with respect to the weights $w_1,w_2,w_3$, the
origin is an isolated singularity of $h$ and $\theta$ contains only
monomials $x^my^nz^l$ such that $w_1m+w_2n+w_3l>d$.

An algebraic surface $S\subset \C^3$ is called {\it
semi-weighted homogeneous} if there exists a semi-weighted homogeneous
polynomial $f=h+\theta$ such that $S=\{(x,y,z)\in \C^3 \ : \
f(x,y,z)=0 \}$. The set $S_0=\{(x,y,z)\in \C^3 \ : \ h(x,y,z)=0\}$
is called a {\it weighted approximation} of $S$.

\begin{theorem}\label{main} Let $S\subset\C^3$ be a semi-weighted homogeneous algebraic surface with an
isolated singularity at origin $0\in\C^3$. If the weights of $S$ satisfy $w_1\geq w_2>w_3$, then
$\upsilon(S,0)\geq 1+\frac{w_2}{w_3}$.
\end{theorem}

\begin{proof}
 Let us consider a family of functions defined as
follows: $$F(X,u)=h(X)+u\theta(X),$$ where $u\in[0,1]$, $X=(x,y,z)$,
and let us denote $F_u(X)=F(X,u)$. Let $V(X,u)$ be the vector field
defined by:
$$V(X,u)=\frac{\theta(X)}{N^{*}F_u(X)}W(X,u)$$
where $$\displaystyle N^{*}F_u(X)= |\frac{\partial F}{\partial
x}(X,u)|^{2\alpha_a}+|\frac{\partial F}{\partial
y}(X,u)|^{2\alpha_b}+|\frac{\partial F}{\partial
z}(X,u)|^{2\alpha_c},$$ $\alpha_a=(d-w_2)(d-w_3)$,
$\alpha_b=(d-w_1)(d-w_3)$, $\alpha_c=(d-w_1)(d-w_2)$ and
$$W(X,u)=W_x(X,u)\frac{\partial }{\partial
x}+W_y(X,u)\frac{\partial}{\partial y}+W_z(X,u)\frac{\partial}{\partial z}$$ where 
$W_x(X,u)=|\frac{\partial
F}{\partial x}(X,u)|^{2\alpha_a-2}\overline{\frac{\partial
F}{\partial x}(X,u)}$, $W_y(X,u)=|\frac{\partial
F}{\partial y}(X,u)|^{2\alpha_b-2}\overline{\frac{\partial
F}{\partial y}(X,u)}$ and $W_z(X,u)=|\frac{\partial
F}{\partial z}(X,u)|^{2\alpha_c-2}\overline{\frac{\partial
F}{\partial z}(X,u)}.$

It was shown, by L. Fukui and L. Paunescu \cite{FP}, that the flow of this vector field gives a modified
analytic trivialization \cite{FP} of the family $F^{-1}(0)$. In
particular, the map defined by
$$\Phi(X)=X+\int_{0}^{1}V(X,u)du$$
is a homeomorphism between $(S,0)$ and $(\tilde{S},0)$ which defines
a correspondence between subanalytic arcs in $(S,0)$ and in
$(\tilde{S},0)$.

  \begin{proposition}\label{estimative_proposition} Let $\gamma\colon [0,\epsilon)\rightarrow V$; $\gamma(0)=0$, be a subanalytic continuous arc parameterized by $\gamma(t)=(t^{w_1}x(t),t^{w_2}y(t),t^{w_3}z(t))$.
  Then $\Phi(\gamma(t))=\gamma(t)+\eta(t)$ such that $\eta(t)=(\eta_1(t),\eta_2(t),\eta_3(t))$ and $\nu(|\eta_i|)>w_i$ for $i=1,2,3$.
  \end{proposition}

\begin{proof}[Proof of the proposition.]
 It is easy to see that there exists a constant
$\lambda>0$ such that $\displaystyle|\frac{\partial F_u}{\partial
x}(\gamma(t))| \leqslant\lambda t^{d-a}$,
$\displaystyle|\frac{\partial F_u}{\partial y}(\gamma(t))|
\leqslant\lambda t^{d-b}$ and $\displaystyle|\frac{\partial
F_u}{\partial z}(\gamma(t))| \leqslant\lambda t^{d-c}$.

It was shown, by M. Ruas and M. Saia (\cite{RS}, Lemma 3, p.93), that
there exists a constant $\lambda_1$ such that
$$N^{*}F_u(\gamma(t))\geqslant\lambda_1N^{*}h(\gamma(t)).$$ Since
$N^{*}h(\gamma(t))=t^{2k}N^{*}h(x(t),y(t),z(t))$, there exists a
constant $\lambda_2$ such that
$$N^{*}F_u(\gamma(t))\geqslant\lambda_2t^{2k}.$$

By hypothesis, $\displaystyle\lim_{t\to
0+}\theta(\gamma(t))t^{-d}=0$. Thus,

\begin{eqnarray*}
|\eta_1(t)|t^{-a} &\leqslant& t^{-a}\int_{0}^{1}
\frac{|\theta(\gamma(t))|}{N^{*}F_u(\gamma(t))}|\frac{\partial
F_u}{\partial x}(\gamma(t))|^{2\alpha_a-1}du \\
&\leqslant&
\frac{{\lambda}^{2\alpha_a-1}}{\lambda_2}|\theta(\gamma(t))|t^{-d}
\end{eqnarray*} $\therefore$
$\displaystyle\lim_{t\to 0+}\eta_1(t)t^{-a}=0$;

\begin{eqnarray*}
|\eta_2(t)|t^{-b} &\leqslant& t^{-b}\int_{0}^{1}
\frac{|\theta(\gamma(t))|}{N^{*}F_u(\gamma(t))}|\frac{\partial
F_u}{\partial x}(\gamma(t))|^{2\alpha_b-1}du \\
&\leqslant&
\frac{{\lambda}^{2\alpha_b-1}}{\lambda_2}|\theta(\gamma(t))|t^{-d}
\end{eqnarray*} $\therefore$
$\displaystyle\lim_{t\to 0+}\eta_2(t)t^{-b}=0$;

\begin{eqnarray*}
|\eta_3(t)|t^{-c} &\leqslant& t^{-c}\int_{0}^{1}
\frac{|\theta(\gamma(t))|}{N^{*}F_u(\gamma(t))}|\frac{\partial
F_u}{\partial x}(\gamma(t))|^{2\alpha_c-1}du \\
&\leqslant&
\frac{{\lambda}^{2\alpha_c-1}}{\lambda_2}|\theta(\gamma(t))|t^{-d}
\end{eqnarray*} $\therefore$
$\displaystyle\lim_{t\to 0+}\eta_3(t)t^{-c}=0$.
\end{proof}

  According to Lemma 1 of \cite{BFN}, we can take an essential loop $\Gamma$ from $S^1$ to the link of the weighted homogeneous approximation of $(X,0)$ of the form:
  $$\Gamma(\theta)=(x(\theta),y(\theta),1).$$
  Then, $H\colon [0,1]\times S^1\rightarrow X$ defined by
  $$H(r,\theta)=\Phi(r^{\frac{w_1}{w_3}}x(\theta),r^{\frac{w_2}{w_3}}y(\theta),r)$$
  is a subanalytic homotopy satisfying:
  $H(0,\theta)=x$ and $H(1,\theta)=\Gamma(\theta)$. We are going to show the image of $H$ ($Im(H)=\Omega$) has volume growth number at origin bigger than or equal to $1+\frac{w_2}{w_3}$. Actually, since the volume growth number of $\Omega$ at $0$ is $1+\beta(\Omega,0)$, we are going to show that $\beta(\Omega,0)$ is bigger that or equal to $\frac{w_2}{w_3}$. So, let us consider two arcs $\gamma_1$ and $\gamma_2$ on $(\Omega,0)$. We can  parameterize these arcs in the following way:  $$\gamma_i(t)=H(t,\theta_i(t)), \ i=1,2.$$ By Lemma \ref{comparison2},
  we have $$\lambda(\gamma_1,\gamma_2)\geq \nu(|\gamma_1(t)-\gamma_2(t)|)$$ and by Proposition \ref{estimative_proposition} ,
 we have $$\nu(|\gamma_1(t)-\gamma_2(t)|)\geq\frac{w_2}{w_3}.$$ Therefore, we can use to Theorem \ref{calculation_of_beta(horn)} to get $\beta(\Omega,0)\geq\frac{w_2}{w_3}$.

\end{proof}

\begin{corollary} Let $S\subset\C^3$ be a semi-weighted homogeneous algebraic surface with an
isolated singularity at origin $0\in\C^3$. If the two low weights of $S$ are unequal, then
the germ $(S,0)$ does not admit a metrical conical structure.
\end{corollary}

\begin{proof} Let $w_1\geq w_2\geq w_3 >0$ be the weights of $S$. By hypothesis, $w_2>w_3$. It follows from Theorem \ref{main}, $$\nu(S,0)\geq 1+\frac{w_2}{w_3}>2.$$ Finally, by Corollary \ref{invariance_corollary}, $(S,0)$ does not admit a metrical conical structure.
\end{proof}

\end{document}